\documentclass{amsart}

\usepackage{amsmath,amsthm,indentfirst}
\usepackage{amssymb}
\usepackage{amsfonts}

\usepackage{amscd}
\usepackage[latin1]{inputenc}

\theoremstyle{plain}
\newtheorem*{maintheorem*}{Main Theorem}
\newtheorem*{thm*}{Theorem}
\newtheorem*{thma*}{Theorem A}
\newtheorem*{thmaa*}{Theorem A'}
\newtheorem*{thmb*}{Theorem B}
\newtheorem*{thmo*}{Theorem 1.1}
\newtheorem*{thmc*}{Theorem C}
\newtheorem*{thmd*}{Theorem D}
\newtheorem*{thmmi*}{Exponential Mixing}
\newtheorem*{thmf*}{Theorem 4.1}
\newtheorem*{remark*}{Remark}
\newtheorem*{conjecture*}{Conjecture}
\newtheorem*{prop*}{Proposition}
\newtheorem*{lem*}{Basic Lemma}
\newtheorem{thm}{Theorem}[section]
\newtheorem{cor}[thm]{Corollary}
\newtheorem{lem}[thm]{Lemma}
\newtheorem{prop}[thm]{Proposition}

\theoremstyle{definition}

\newtheorem*{proofc*}{Proof of Theorem C}

\newtheorem{remark}[thm]{Remark}

\def\o{\mathcal{O}}
\def\ocal{\mathcal{O}}
\def\ccal{\mathcal{C}}
\def\fcal{\mathcal{F}}
\def\pnd{\mathcal{P}_{N,d}}
\def\bbz{\mathbb{Z}}
\def\bbq{\mathbb{Q}}

\def\bbr{\mathbb{R}}
\def\bba{\mathbb{A}}

\def\bbh{\mathbb{H}}

\def\bbg{\mathbb{G}}

\def\bbu{\mathbb{U}}
\def\bbp{\mathbb{P}}
\def\bbl{\mathbb{L}}
\def\bbv{\mathbb{V}}
\def\bbw{\mathbb{W}}
\def\bbm{\mathbb{M}}

\def\Bfrak{\mathfrak{B}}

\def\scal{\mathcal{S}}
\def\kcal{\mathcal{K}}
\def\vare{\varepsilon}

\def\h{\hspace{1mm}}
\def\hh{\hspace{.5mm}}
\def\cc{C_{00}^\infty}

\def\azs{A_{I}^0}
\def\uas{A^{(1)}}
\def\uais{A_{I}^{(1)}}
\def\supp{{\rm{supp}}}
\def\pnmd{\mathcal{P}_{N,M,d}^*}

\title[Measures invariant under horospherical subgroups]{\sc Measures Invariant Under Horospherical Subgroups in Positive characteristic}
\author{A.~Mohammadi}
\date{}

\address{Mathematics Dept., University of Chicago, Chicago, IL}
\email{amirmo@math.uchicago.edu}

\begin{document}
\maketitle

\begin{abstract}
We prove measure rigidity for the action of (maximal) horospherical subgroups on homogeneous spaces obtained by quotient by a uniform (nonuniform) arithmetic lattices over a field of positive characteristic.
\end{abstract}

\section{Introduction}\label{sec;intro}

Let $K$ be a global function field and let $S$ be a finite set of places in $K.$ Let $\o_S$ be the ring of $S$-integers. For each $\nu\in S$ let $K_{\nu}$ be the completion of $K$ with respect to $\nu$ and let $K_S=\prod_{\nu\in S}K_\nu.$ Let $\bbg$ be a connected semisimple group defined over $K$. Let $G=\prod_{\nu\in S}\bbg(K_{\nu}).$ The locally compact topology of $K_S$ induces a locally compact topology on $G$ which we will refer to as the Hausdorff topology. Let $\Gamma$ be a congruence lattice in $\bbg(\o_S).$ We let $X=G/\Gamma.$ By a horospherical subgroup we mean the unipotent radical of a $K_S$-parabolic subgroup of $G.$ If $\bbp$ is a $K$-parabolic of $G,$ let $\bbp=\bbl\bbw$ be the Levi decomposition of $\bbp$ and let $\bbm=[\bbl,\bbl].$ Let ${}^\circ P=MW$ where $M=\bbm(K_S)$ and $W=\bbw(K_S).$ We let $P^+$ denote the group generated by all $K_S$-split unipotent subgroups of $P=\bbp(K_S),$ see section~\ref{sec;notation} for more details. In this paper we prove

\begin{thm}\label{horospherical}
Let notation be as above. Assume $U$ is a horospherical subgroup of $G.$ Let $\mu$ be a $U$-invariant ergodic probability measure on $X=G/\Gamma.$ 
\begin{itemize}
\item[(i)] If $\Gamma$ is a uniform lattice, then the action of $U$ is uniquely ergodic.
Furthermore $\mu$ is the probability Haar measure on the closed orbit $\overline{G^+\Gamma}/\Gamma$ where the closure is with respect to the Hausdorff topology.  
\item[(ii)] If $\Gamma$ is non-uniform, then assume that there exists some $\nu\in S$ such that $U$ contains a maximal horospherical subgroup of $\bbg(K_{\nu}).$ There exists some
$K$-parabolic subgroup $\bbp$ of $\bbg$ and some $g\in G$ such that $g^{-1}Ug\subset P=\mathbb{P}(K_S)$ and $\mu$ is the $\Sigma$-invariant measure on the closed orbit $\Sigma\h g\Gamma$ where $\Sigma=g\overline{P^+({}^\circ P\cap\Gamma)}g^{-1}$ and the closure is with respect to the Hausdorff topology.
\end{itemize}
\end{thm}

Using this theorem and the Linearization techniques, see~\cite{DM2}, we conclude, as a corollary, equidistribution of orbits of subgroups satisfying the assumptions of Theorem~\ref{horospherical}. The statement of this corollary needs some notation which will be fixed in section~\ref{sec;notation} hence we postpone the discussion to the end of the paper, see Corollary~\ref{equi-horo}. It is worth mentioning that the method which is used here in order to prove Theorem~\ref{horospherical} seems to actually be enough to get the equidistribution result without appealing to linearization techniques, indeed in the course of the proof of Theorem~\ref{horospherical} we prove equidistribution in case (i) and ``generic" case of (ii), see the proof in section~\ref{sec;horo} for more details.  

These are indeed very special cases of Ratner's measure classification and equidistribution theorems in positive characteristic, see~\cite{Rat3, Rat4, Rat5, Rat6} and also~\cite{MT}. The full measure classification theorem in positive characteristic in not known, there have been some partial results and developments over the past few years. The case of horospherical subgroups is probably ``the easiest" case to consider. This problem in the real case was considered by S.~G.~Dani~\cite{D1, D2} some ten years before Ratner's proof. Our approach is however different from Dani's approach. We use mixing properties of semisimple elements. These lines of ideas go back to G. A. Margulis' PhD thesis and have been extensively utilized by many others ever since. The idea of using mixing property to prove measure classification for horospherical flows, used in this paper, are by no means new however the results in positive characteristic are new, see also~\cite{Rat1}. Another possible approach to this problem would be to use representation theoretic approach of M.~Burger in~\cite{Bu}. It is worth mentioning that representaion theoretical  proofs work very well when the horospherical subgroup is maximal. This assumtion seems inavitable in both Dani's and Burger's proof. Our approach treats the case: $\Gamma$ is a uniform lattice, without this assumption and actually the proof in this case is very simple. In the non-uniform case we too also need assume that the horspherical subgroup contains the maximal horospherical subgroup in one place. Indeed over local fields with characteristic zero the case: $\Gamma$ is non-uniform and the horospherical is not maximal follows from measure rigidity theorem. What is surprising is that, to the best of our knowledge, no simpler proof is known even for this very special case.    

\textbf{Acknowledgements.} We would like to thank Professor G.~A.~Margulis for leading us into this direction of research and for many enlightening conversations in the course of our graduate studies and also after graduation to date. We would also like to thank M. Einsiedler for his interest in this project and his encouragement.       

\section{Notation and Preliminary}\label{sec;notation}

Let $G$ be as above. There are two different topologies on $G,$ the Zariski topology and the locally compact topology which is induced from the topology of $K_S.$ We will refer to the locally compact topology as the Hausdorff topology. Let $\kcal$ be a good maximal compact subgroup of $G$ in the sense of~\cite[section 3]{Ti}, see also~\cite[proposition 2.1]{Oh}. In particular the Iwasawa decomposition and the Cartan decomposition for $G$ holds with $\kcal,$ as the maximal compact subgroup. We fix a left $G$-invariant bi-$\kcal$-invariant metric $d$ on $G$ and let $\|g\|=d(1, g).$ We let 
$C^{\infty}(X)$ denote the space of functions $\phi:X\rightarrow\bbr$ such that there exists a compact open subgroup $\kcal_\phi$ of $G$ which leaves $\phi$ invariant. 
Note that these functions are indeed $\kcal$-finite.   

Let $s\in G$ be a diagonalizable element over $K_{S}$ such that the $\nu$-component $s_{\nu}$ of $s$ has eigenvalues which are integer powers of the uniformizer $\varpi_{\nu},$ of $K_\nu$ for all $\nu\in S.$ We call such $s$ ``an element from class $\mathcal{A}.$". Define

$$W_G^+(s)=\{x\in G\hh|\h s^{n}xs^{-n}\rightarrow e\h\h\mbox{as}\h\h n\rightarrow-\infty\}$$
$$\begin{array}{l}W_G^-(s)=\{x\in G\hh|\h s^{n}xs^{-n}\rightarrow e\h\h\mbox{as}\h\h n\rightarrow+\infty\} \vspace{.75mm}\\ Z_G(s)=\{x\in G\hh|\h sxs^{-1}=x\}\end{array}$$

If it is clear from the context we sometimes omit the subscript $G$ above. Note that $W_G^{\pm}(s)$ and $Z_G(s)$ are the groups of $K_S$-points of $K_S$-algebraic subgroups $\mathbb{W}_G^{\pm}(s)$ and $\mathbb{Z}_G(s)$ respectively. The product map is a $K_S$-isomorphism of $K_S$-varieties between  $\mathbb{W}^-(s)\times\mathbb{Z}_G(s)\times\mathbb{W}^+(s)$ and a Zariski open dense subset of $\bbg$ containing the identity, see for example~\cite[chapter V]{B1} for these statements. Hence if $\kcal^{\ell}$ is a deep enough congruence subgroup then 
\begin{equation}\label{e;iwasawa-comp}\kcal^\ell=(\kcal^\ell\cap W^-(s))((\kcal^\ell\cap Z_G(s)))((\kcal^\ell\cap W^+(s)))=(\kcal^\ell)^-(\kcal^\ell)^0(\kcal^\ell)^+\end{equation}
In the sequel if $\bullet_\nu$ denotes the $\nu$-component of $\bullet.$ Let $s$ be an element from class $\mathcal{A}$ and let $U\subset W^+(s)$ be a unipotent subgroup which is normalized by $s$. Let $U_0=U\cap\mathcal{K}.$ For any $m>0$ define $U_m=s^mU_0s^{-m}$ and $(U_\nu)_m=s^m(U_\nu)_0s^{-m}.$ These form a filtration of $U.$ We let $\theta_\nu$ be a Haar measure on $(U_\nu)_0$ normalized such that $\theta_\nu((U_\nu)_0)=1$ and let $\theta=\prod_\nu\theta_\nu.$ Note that $\{U_m\}$ and $\{(U_\nu)_m\}$ form an averaging sequence for $U$ and $U_\nu$ respectively.

Let $\bbh$ be a group defined over $K_S$ and let $H=\bbh(K_S).$ We will let $H^+$ denote the group generated by all $K_S$-points of all the $K_S$-split unipotent subgroups of $H.$ Thanks to~\cite[theorem, C.3.8]{CGP} we have $H^+$ is the group generated by the $K_S$-points of all the $K_S$-split unipotent radicals of minimal $K_S$-pseudo parabolic subgroups of $\bbh.$ Note that $H^+$ is a normal unimodular subgroup of $H.$ 

Let $\mathbb{H}$ be a $K_S$-algebraic group acting $K_S$-rationally on a $K_S$-algebraic variety $\mathbb{M}$. Assume that $H=\bbh(K_S)$ is generated by one parameter $K_S$-split unipotent algebraic subgroups and elements from class $\mathcal{A}.$ The following is proved in~\cite{MT}. 

\begin{lem}\label{h;bz-measure}(cf.~\cite[Lemma 3.1]{MT})
Let $\mu$ be an $H$-invariant Borel probability measure on $M=\mathbb{M}(K_S)$. Then $\mu$ is concentrated on the set of $H$-fixed points in $M.$ In particular if $\mu$ is $H$-ergodic then $\mu$ is concentrated at one point.   
\end{lem}
It is worth mentioning that the proof in~\cite{MT} follows from the fact that the orbit space in this situation is Hausdorff which is a consequence of~\cite[appendix A]{BZ}).

We also need the following general and simple statement.

\begin{lem}\label{h;normal-unimodular}(cf.~\cite[Lemma 10.1]{MT})
Let $A$ be a locally compact second countable group and $\Lambda$ a discrete subgroup of $A.$ If $B$ is a normal unimodular subgroup of $A$ and $\mu$ is a $B$-invariant ergodic measure on $A/\Lambda.$ Let $\Sigma=\overline{B\Lambda},$ where the closure is with respect to the Hausdorff topology. There exists $x\in A/\Lambda$ such that $\Sigma\hh x$ is closed and $\mu$ is the $\Sigma$-invariant Haar measure on $\Sigma\hh x.$ 
\end{lem}

Let $L^2_{00}(X)$ denote the orthogonal complement of $G$-characters appearing in the regular representation of $G$ on $L^2(X).$ Hence we have $L^2(X)=L^2_{00}(X)\oplus(L^2_{00}(X))^{\perp}.$ Indeed $L^2_{00}(X)$ can also be characterized as the orthogonal complement of $G^+$-fixed vectors, we will actually use this characterization in this paper. We let $\cc(X)=C^{\infty}(X)\cap L^2_{00}(X).$ We recall the following important and by now folk statement. 

\begin{thm}\label{thm;exp-mixing}(Exponential Decay of Matrix Coefficients)

Let $G$ and $\Gamma$ be as above. Then there exist positive constants $\kappa$ and $E$ depending on $G$ and $\Gamma$ such that for all $\phi,\psi\in\cc(X)$
\begin{equation}\label{e;exp-mixing}\left|\langle g\phi,\psi\rangle\right|\leq Ee^{-\kappa\|g\|}(\dim\langle\mathcal{K}\phi\rangle\dim\langle\mathcal{K}\psi\rangle)^{\frac{1}{2}}\end{equation}
where $\langle\mathcal{K}\bullet\rangle$ denotes the $\mathcal{K}$-span of $\bullet.$
\end{thm}
If for all $\nu\in S$ the $K_\nu$-rank of all of the $K_\nu$-simple factors of $\bbg$ is at least 2, then the above is a result of Kazhdan property (T). In this case more is true in particular in this case the constant $\kappa$ is independent of $\Gamma,$ see~\cite{Oh}. In general this theorem follows from~\cite{Cl} and~\cite{BS} in the number field case. It is known to experts that the function field case works out much the same way, however we were not able to find an account on this in the literature. 

The following is a consequence of the above theorem. 

\begin{prop}\label{m;equidistribution}
Let $s\in G^+$ be an element from class $\mathcal{A}$ and let $U=W_G^+(s).$ Then for any $f\in C_c^{\infty}(U),$ compactly supported locally constant function on $U,$ for any $\phi\in \cc(X)$ and any compact set $L\subset X$ there exists a constant $C=C(f,\phi, L)$ such that for all $x\in L$ and for any $n\geq0$ we have 
\begin{equation}\label{e;equi}\left|\int_Uf(u)\phi(s^nux)d\theta(u)\right|\leq Ce^{-n\kappa'}\end{equation} 
\end{prop}

\begin{proof}
This is proved in~\cite{KM1} in the real case and the proof works in $S$-arithmetic setting also, we recall the proof in the positive characteristic case for the sake of completeness. Note that there is a compact open subgroup $U_f$ (resp. $\mathcal{K}_{\phi}$) of $U,$ (resp. $G$) which leaves $f$ (resp. $\phi$) invariant. 

As was mentioned above if $g\in G$ is close enough to the identity then $g=u^-u^0u^+$ where $u^\pm\in W_G^\pm(s)$ and $u^0\in Z_G(s).$ For $\ell$ large enough let $\kcal^\ell$ denote the $\ell$-th congruence subgroup of $\kcal$ and let $(\kcal^\ell)^\pm=W_G^\pm(s)\cap\kcal^\ell$ and $(\kcal^\ell)^0=Z_G(s)\cap\kcal^\ell.$  

Note that it suffices to prove the statement for $f$ which is the characteristic function of $U_f,$ we assume this is the case from now. 
Now since $L$ and $\kcal_\phi$ are compact by taking $\ell$ large, we may assume that the map from the $\ell=\ell(L,\kcal_\phi)$-th congruence subgroup $\mathcal{K}^\ell$ of $\mathcal{K}$ into $X$ is injective at $x$ for all $x\in L$ and that $\mathcal{K}^\ell\subset\mathcal{K}_{\phi}.$ Replacing $U_f$ by $U_f\cap(\kcal)^\ell$ we may and will assume that $U_f=(\mathcal{K}^{\ell})^+.$
Hence the map from $\mbox{supp}(f)$ to $X$ is injective at all $x\in L$. 

Let $f^-=\frac{1}{|(\mathcal{K}^{\ell})^-|}\chi_{(\mathcal{K}^{\ell})^-}$ and  $f^0=\frac{1}{|(\mathcal{K}^{\ell})^0|}\chi_{(\mathcal{K}^{\ell})^0},$ be the normalized characteristic functions of  $(\mathcal{K}^{\ell})^-$ and $(\mathcal{K}^{\ell})^0$ respectively. Define now 
$$\tilde{f}(u^-u^0)=f^-(u^-)f^0(u^0)$$ 
Let $n$ be a large positive number. Note that the restriction of the module function of the corresponding parabolic $(\mathcal{K}^{\ell})^0$ is trivial, thus we have
\begin{equation}\label{e;bring-u-u0}\left|\int_Uf(u)\phi(s^nux)d\theta(u)-\int_G\tilde{f}(u^-u^0)f(u)\phi(s^nu^-u^0ux)d(u^-u^0u)\right|=\end{equation}
$$\left|\int_Gf^-(u^-)f^0(u^0)f(u)(\phi(s^nux)-\phi(s^nu^-u^0ux))d(u^-u^0u)\right|=0$$
The last equality follows from the fact that $\kcal^\ell\subset\kcal_\phi.$ Define now $\psi=f^-f^0f$ this is a function supported in a small neighborhood of the identity and thanks to the fact that $\kcal^\ell$ maps injectively to a neighborhood of $x\in L$ for any $x\in L,$ we may also define $\psi_x(gx)=\psi(g).$ This is a (compactly supported) function in $C^{\infty}(X).$ Let ${(\psi_x)}_{00}\in\cc(X)$ be the component of $\psi_x$ in $L^2_{00}(X).$ We now use \eqref{e;bring-u-u0} and Theorem~\ref{thm;exp-mixing} and get 
$$\left|\int_Uf(u)\phi(s^nux)d\theta(u)\right|=\left|\langle s^n\phi,\psi_x\rangle\right|=|\langle s^n\phi,{(\psi_x)}_{00}\rangle|$$
$$\leq Ee^{-\kappa\|s^n\|}(\dim\langle\mathcal{K}\phi\rangle\dim\langle\mathcal{K}(\psi_x)_{00}\rangle)^{\frac{1}{2}}\leq E\h[\mathcal{K}:\mathcal{K^\ell}]\h e^{-n\kappa'}$$
The proposition is proved. 
\end{proof}

\begin{remark}\label{non-quantitaive}
Here we stated Theorem~\ref{thm;exp-mixing} and proved Proposition~\ref{m;equidistribution} in the quantitative form. It is worth mentioning that the qualitative version of Theorem~\ref{thm;exp-mixing} i.e. Howe-Moore theorem which is a consequence of Mautner phenomena can be found in the literature, see~\cite{BO} and references in there. Indeed one can then state and prove qualitative version of this proposition assuming only a qualitative version of Theorem~\ref{thm;exp-mixing}. In other words with notation as above for any $\vare>0$ there exits $n>0$ such that
\begin{equation}\label{e;equi'}\left|\int_Uf(u)\phi(s^nux)d\theta(u)\right|\leq\vare\end{equation} 

This qualitative statement suffices for the proof of Theorem~\ref{horospherical} in the way it is stated here. 
\end{remark}

\section{Non-divergence of unipotent flows} 
In the proof of Theorem~\ref{horospherical} we will need results concerning non-divergence of unipotent orbits in $X.$ Results of this sort were first established by Margulis' in the course of the proof of arithmeticity of non-uniform lattices. Later Dani, Dani-Margulis and Kleinbock-Margulis proved quantitative versions of these non-divergence result. We need an $S$-arithmetic version of~\cite{DM1}. In characteristic zero such $S$-arithmetic statements are proved in~\cite[Section 7]{GO}. Their proof, which is modeled on~\cite{DM1}, essentially works in our setting as well. However there are several points which require some explanation. For the sake of completeness we will reproduce the proof in positive characteristic in this section. 

Let the notation be as in the introduction. Let $\mathbb{A}$ be a maximal $K$-split torus of $\bbg$ and choose a system $\{\alpha_1,\ldots,\alpha_r\}$ of simple $K$-roots for $(\bbg,\mathbb{A})$. 
For each $1\leq i\leq r$ let $\mathbb{P}_i$ be the standard maximal parabolic subgroup of $\bbg$ corresponding to $\alpha_i.$ 
The subgroup $\mathbb{P}=\cap_{i}\mathbb{P}_i$ is a minimal $K$-parabolic subgroup of $\bbg.$ 

We need some results from reduction theory, see~\cite{Sp} and~\cite{Ha} for general results concerning reduction theory in the function field case. In particular we need the following which is Theorem D in~\cite{Be}, we thank A. Salehi Golsefidy for pointing out this reference to us. It is worth mentioning that this could also be proved using~\cite[Proposition 2.6]{Sp}. 
\begin{equation}\label{e;cusps}\mbox{There exists a finite set}\h\h F\subset\bbg(K)\h\h\mbox{such that}\h\h\bbg(K)=\mathbb{P}(K)F\hh\Gamma \end{equation}

Through out this section $U=U_\nu$ will denote a $K_\nu$-split unipotent subgroup of $G.$ We may and will assume $U\subset W_G^+(s)$ where $s$ is an element from class $\mathcal{A}.$ We will further assume that $U$ is normalized by $s.$ As before let $U_0=U\cap\kcal$ and define $U_m=s^mU_0s^{-m}.$ We will assume that $U$ is an $N$-dimensional $K_\nu$-group, thus there is an $s$-equivariant isomorphism of $K_\nu$-varieties between the Lie algebra of $\mathfrak{u}$ of $U$ onto $U,$ see~\cite[corollary, 9.12]{BS}. We will let $\Bfrak$ be the image of $U_0$ under this isomorphism. Let $\theta$ denote the Haar measure on $U$ normalized so that $\theta(U_0)=1.$ Abusing the notation we will let $\theta$ also denote the Haar measure on $K_\nu^N$ normalized so that $\theta(\Bfrak)=1.$ Let 
$$\pnd=\left\{\Theta: \Bfrak\rightarrow G:\h \Theta\h\mbox{is a polynomial map with degree at most}\h d\right\}$$ 
Note that there exists some $d$ such that for any $g\in G$ the map $u\mapsto ug$ is in $\pnd.$

The following is the main result of this section, see Theorem 2~\cite{DM1} and Theorem 7.1~\cite{GO}. 

\begin{thm}\label{non-divergence}
Let the notation be as above and let $\vare>0.$ There exists a compact subset $L\subset X$ such that for all $x=g\Gamma\in X$ one of the following holds
\begin{enumerate}
\item for all large $m,$ depending on $x,$
$$\theta(\{u\in U_m\h:\h ux\in L\})\geq (1-\vare)\theta(U_m)$$
\item there is some $\lambda\in F\hh\Gamma,$ such that $g^{-1}Ug\subset \lambda \mathbb{P}_i\lambda^{-1}.$
\end{enumerate} 
\end{thm}

We will first make several reductions. Note that if we let $p:{\bbg}\rightarrow\overline{\bbg}$ be the $K$-central isogeny from $\bbg$ to the adgoint form $\overline{\bbg}$ of $\bbg,$ then  
the natural map from $G/\Gamma$ to $\overline{\bbg}(K_S)/\overline{\bbg}(\mathcal{O}_S)$ is a proper map. Thus the statement and the conclusion of Theorem~\ref{non-divergence} are not changed if we replace $\bbg$ by the adjoint form. Hence for the rest of this section we will assume $\bbg$ is adjoint form.  

We need some more notation. Recall that $\{\alpha_1,\ldots,\alpha_r\}$ is the set of simple simple $K$-roots for $(\bbg,\mathbb{A}).$ Since $\bbg$ is adjoint form they coincide with the fundamental weights. Note now that $\alpha_i$'s uniquely extend to a character of $\bbp_i$ which we will also denote by $\alpha_i.$ For any $1\leq i\leq r$ we fix, once and for all, representations $(\rho_i, \mathbb{V}_i, v_i)$ of $\bbg$ defined over $K$ and vectors $v_i\in \mathbb{V}_i(K)$ such that there are $K$-rational characters $\chi_i$'s so that 
\begin{equation}\label{e;chevalley}\bbp_i=\{g\in\bbg:\h gv_i=\chi_i(g)v_i\}\end{equation} 
We further assume $\chi_i=\alpha_i^{n_i}$ where $n_i$ is a positive integer.
We fix $\kcal_\nu$-invariant norms $\|\h\|_\nu$ on $\bbv_i(K_S)$ and assume $\|v_i\|_\nu=1$ for all $\nu\in S$ and all $1\leq i\leq r.$ Define $\Phi_i:G\rightarrow\bbr^\times$ by
\begin{equation}\label{e;standard-map}\Phi_i((g_\nu))=\prod_{\nu\in S}\|g_\nu v_i\|_\nu \end{equation} 
Since $\Phi_i(\hh)$ is left $\kcal$-invariant for all $i,$ it only depends on $P_i$ component of $g$ in the Iwasawa decomposition. Similar to~\cite{Sp}, representations $\rho_i$'s and $\Phi_i$'s will be used as `the standard representation of $\bbg$ corresponding to $\bbp_i$" which are used in reduction theory and are also extensively used in~\cite{DM1} and~\cite{GO}.   

We have the following

\begin{thm}\label{non-div-poly}
For any $\vare>0$ and $\alpha>0$ there exists a compact subset $L$ of $G/\Gamma$ such that for any $\Theta\in\pnd$ one of the following holds
\begin{enumerate}
\item $\theta(\{t\in \Bfrak\h:\h \Theta(t)\Gamma/\Gamma\in L\})\geq (1-\vare)$\vspace{1mm}
\item there exist $1\leq i\leq r$ and $\lambda\in F\Gamma$ such that $\Phi_i(\Theta(t)\lambda)<\alpha$ for all $t\in\Bfrak.$
\end{enumerate}
\end{thm}

{\bf Proof of Theorem~\ref{non-divergence} module Theorem~\ref{non-div-poly}.} Let $\vare$ and $\alpha$ be given and let $L$ be the compact set obtained as in the Theorem~\ref{non-div-poly} for this $\vare$ and $\alpha.$ Assume now that (1) in the Theorem~\ref{non-divergence} does not hold, thus there exists a sequence $m_j\rightarrow\infty$ such that loc. cit. fails for $U_{m_j}.$ As we mentioned before there exists some $d$ such that the map $u\mapsto ug$ is in $\pnd$ for all $U_{m_j}.$ We denote these maps by $\Theta_j.$ Now  Theorem~\ref{non-div-poly} implies that there exist some $1\leq i\leq r$ and $\lambda_j\in F\Gamma$ such that $\Phi_i(\Theta_j(t)\lambda_j)<\alpha$ for all $t\in\Bfrak$ and all $j$ i.e. if $u\in U_{m_j}$ then $\Phi_i(ug\lambda_j)<\alpha.$ In particular we have  $\Phi_i(g\lambda_j)<\alpha$ for all $j.$ Note also that $gF\Gamma v_i$ is a discrete set hence there are only finitely many $\lambda$'s module the stabilizer of $v_i$ such that $\Phi_i(ug\lambda)<\alpha.$ Since $U_{m_j}$'s exhaust $U$ the above discussion implies that there exits some $\lambda\in F\hh\Gamma$ such that $\Phi_i(ug\lambda)<\alpha$ for all $u\in U.$
If we utilize the isomorphism $f$ between $U$ and ${K_\nu}^N,$ then the map $\rho_i(ug\lambda)v_i$ is a polynomial map from $K_\nu^N$ into $V_i.$ Hence either it is unbounded or constant. Since $\Phi_i(ug\lambda)<\alpha$ we get $\rho_i(ug\lambda)v_i$ is constant which using ~\eqref{e;chevalley} implies that $\lambda^{-1}g^{-1}Ug\lambda\subset \bbp_i.$

Let us recall the following quantitative non-divergence result.

\begin{thm}\label{kleinbock-margulis}(cf.~\cite[Theorem, 4.3]{Gh}) 
For any given compact subset $L\subset G/\Gamma$ and $\vare>0$ there exists a compact subset $L'\subset G/\Gamma$ such that for any $\Theta\in\pnd$ and any $y\in G/\Gamma$ such that $\Theta(\Bfrak)y\cap L\neq\emptyset$ we have
$$\theta(\{t\in \Bfrak:\h \Theta(t)y\in L'\})\geq (1-\vare)$$ 
\end{thm}

Using Theorem~\ref{kleinbock-margulis} the proof of Theorem~\ref{non-div-poly} reduces to the following 

\begin{thm}\label{non-div}
For any $\alpha>0$ there exists a compact subset $L$ of $G/\Gamma$ such that for any $\Theta\in\pnd$ one of the following holds
\begin{enumerate}
\item $\Theta(\Bfrak)\Gamma/\Gamma\cap L\neq\emptyset$\vspace{1mm}
\item there exist $1\leq i\leq r$ and $\lambda\in F\Gamma$ such that $\Phi_i(\Theta(t)\lambda)<\alpha$ for all $t\in\Bfrak.$
\end{enumerate}
\end{thm}

The proof of Theorem~\ref{non-div} will occupy the rest of this section. Some more notation is needed. For any $1\leq i\leq r$ we let $\bbq_i=\{p\in\bbp_i:\h\alpha_i(p)=1\}$ and let $\bba_i=\{a\in\bba:\h\alpha_j(a)=1, \forall j\neq i\}.$ For any subset $I\subset\{1,2\ldots,r\}$ we let
$$\bbp_I=\cap_{i\in I}\bbp_i,\h\h \bbq_I=\cap_{i\in I}\bbq_i,\h\h \bba_I=\prod_{i\in I}\bba_i$$
$\bbq_I$ is a normal $K$-subgroup of $\bbp_I,$ $\bba_I$ is a $K$-split torus and $\bbp_I=\bba_I\bbq_I.$ Let $\bbu_I$ be the unipotent radical of $\bbp_I$ and let $\bbh_I$ be the centralizer of $\bba_I$ in $\bbq_I.$ We have $\bbq_I=\bbh_I\bbu_I$ and $\bbh_I$ and $\bbu_I$ are defined over $K$ and $\bbu_I$ is $K$-split. As usual we let $P_I=\bbp_I(K_S),\h Q_I=\bbq_I(K_S),\h A_I=\bba_I(K_S)\h H_I=\bbh_I(K_S),\h U_I=\bbu_I(K_S).$ 

Since $\bbg$ is adjoint for any subset $I\subset\{1,2\ldots,r\}$ we have  
\begin{equation}\label{e;weight-isomo}\mbox{The map}\h\h{(\alpha_i)}_{i=1}^r:\bba_I\rightarrow\bbg_m^{|I|}\h\h\mbox{is an isomorphism of}\h\h K\mbox{-varieties}\end{equation} 
For any $\nu\in S$ fix a uniformizer $\varpi_\nu\in K\cap \ocal_\nu$ and let $$A_\nu^0=\{a\in\bba(K_\nu):\h \alpha_i(a)\in\varpi_\nu^\bbz,\h\h\mbox{for all}\h i=1,\ldots,r\}$$
Note that by our choice of $\varpi_\nu$ and from ~\eqref{e;weight-isomo} we get that $A_\nu^0\subset\bba(K)$ for all $\nu\in S.$ Let $A^0=\prod_{\nu\in S}A_\nu^0.$ For any subset $I\subset\{1,2\ldots,r\}$ we let 
$\azs=A_I\cap A^0.$ We may and will assume that $\kcal$ is chosen so that $A=\kcal\azs.$ We let 
$$\uas=\left\{a\in A^0:\h \Phi_i(a)=1,\h\h\mbox{for all}\h i=1,\ldots,r\right\}$$
and let $\uais=\uas\cap A_I.$

Recall that $\Gamma$ is a congruence subgroup of $\bbg(\ocal_S).$ Hence the choice of $\varpi_\nu,$ the fact that $\chi_i=\alpha_i^{n_i}$ and~\eqref{e;weight-isomo} imply that
\begin{equation}\label{e;azn-compact}\mbox{There is a compact subset}\h\h Y\subset A_I\h\h\mbox{such that}\h\h \uais\subset Y(A_I\cap\Gamma)\end{equation}
We have the following
\begin{lem}\label{go-7.6}
Let $I\subset\{1,2,\ldots,r\},$ $j\in\{1,2,\ldots,r\}\setminus I$ and $0<a<b$ be given. Then there exists a compact subset $M_0$ of $Q_I$ such that
$$\{g\in Q_I:\h \Phi_j(g)\in[a,b]\}\subset M_0Q_{I\cup\{j\}}(A_j\cap\Gamma)$$
\end{lem}

\begin{proof}
Note that $H_I=(\kcal\cap H_I)(P_j\cap H_I).$ Hence we have $Q_I=H_IU_I=(\kcal\cap H_I)(P_j\cap H_I)U_I\subset (\kcal\cap Q_I)(P_j\cap Q_I).$ Note also that $P_j\cap Q_I=A_jQ_{j\cup I}.$ Thus we have $Q_I=(\kcal\cap Q_I)A_jQ_{j\cup I}.$ Now let $g\in Q_I$ such that $\Phi_j(g)\in[a,b]$ then we have $q=kaq$ where $k\in\kcal\cap Q_I,\h a\in A_j,\h q\in Q_{j\cup I}.$ Hence we have $\Phi_j(g)=\Phi_j(a)\in [a,b]$ which is to say $g\in(\kcal\cap Q_I)\{a\in A_j:\h \Phi_j(a)\in[a,b]\}Q_{j\cup I}.$ Note that there exists a compact subset $Y'\subset A_j$ such that $\{a\in A_j:\h \Phi_j(a)\in[a,b]\}\subset Y'A_{j}^{(1)}.$ Hence using~\eqref{e;azn-compact} we have $\{a\in A_j:\h \Phi_j(a)\in[a,b]\}\subset Y'Y(A_j\cap\Gamma)$ where $Y$ is as in~\eqref{e;azn-compact}. Since $A_j$ normalizes $Q_{j\cup I}$ the lemma thus follows with $M_0=(\kcal\cap Q_I)Y'Y.$
\end{proof}

For $I\subset\{1,\ldots,r\}$ there is a finite subset $\tilde{F}_I\subset\bbq_I(K)$ such that 
\begin{equation}\label{qi-cusps}\bbq_I(K)=(\bbp\cap\bbq_I)(K)\tilde{F}_I(Q_I\cap\Gamma)\end{equation} 
Recall that $\bba_I$ normalizes $\bbq_I$ and elements with bounded denominators in $\bbq_I(K)$ lie in finitely many cosets of $Q_I\cap \Gamma.$ Thus we may find a finite subset $F_I\subset\bbq_I(K)$ such that
\begin{equation}\label{lambdai1}(A_I\cap\Gamma)(Q_I\cap\Gamma)\tilde{F}_I^{-1}\subset(Q_I\cap\Gamma)(A_I\cap\Gamma)\tilde{F}_I^{-1}\subset(Q_I\cap\Gamma)F_I^{-1}(A_I\cap\Gamma)\end{equation}
Let
\begin{equation}\label{lambdai1}\Lambda(I)=(Q_I\cap\Gamma)F_I^{-1}\subset\bbq_I(K)\end{equation}
Note that $\bbp_\emptyset=\bbq_\emptyset=\bbg,\h \bba_\emptyset=\bba,\h \tilde{F}_\emptyset=F={F}_\emptyset$ and $\Lambda(\emptyset)=\Gamma F^{-1}.$
\begin{lem}\label{7.8go}
For $j\in\{1,\ldots,r\}$ and $I\subset\{1,\ldots,r\}\setminus\{j\}$ there exists a finite subset $E\subset\bbp(K)$ such that $\Lambda(I)\Lambda(I\cup\{j\})\subset\Lambda(I)E.$
\end{lem}
\begin{proof}
This is a consequence of the definition together with~\eqref{qi-cusps}.
\end{proof}

We now construct certain compact subsets of $G/\Gamma$ using the above. These will eventually give us the set $L$ which we need in the proof of Theorem~\ref{non-div}.

An $l$-tuple $((i_1,\lambda_1),\ldots,(i_l,\lambda_l))$ where $l\geq1,$ $i_1\ldots,i_l\in\{1,\ldots,r\}$ and $\lambda_1,\ldots\lambda_l\in\bbg(K)$ is called an called an {\it admissible} sequence of length $l$ if $i_1\ldots,i_l$ are distinct and $\lambda_{j-1}^{-1}\lambda_j\in\Lambda(\{i_1,\ldots,i_{j-1}\})$ for all $j=1,\ldots,l$ and we let $\lambda_0$ be the identity element. The empty sequence is called an admissible sequence of length $0.$ If $\xi$ and $\eta$ are two admissible sequences of length $l$ and $l'$ with $l\leq l'$ we say $\eta$ extends $\xi$ if the first $l$ terms of $\eta$ coincide with $\xi.$ For an admissible sequence $\xi$ of length $l\geq0$ we let $\ccal(\xi)$ be the set of all pairs $(i,\lambda)$ with $1\leq i\leq r$ and $\lambda\in\bbg(K)$ for which there exists an admissible sequence $\eta$ of length $l+1$ extending $\xi,$ for such $\eta$ the pair $(i,\lambda)$ is necessarily the last term. Note that if $l=0$ then $\ccal(\xi)$ consists of all pairs $(i,\lambda)$ with $1\leq i\leq r$ and $\lambda\in\Lambda(\emptyset).$

For any admissible sequence $\xi$ of length $l$ we define the support of $\xi,$ which we will denote by $\supp(\xi),$ to be the empty set of $l=0$ and the set $\{(i_1,\lambda_1),\ldots,(i_l,\lambda_l)\}$ if $\xi=((i_1,\lambda_1),\ldots,(i_l,\lambda_l)).$

Let $\xi$ be an admissible sequence of length $l\geq0.$ Let $\alpha$ and $a<b$ be positive real numbers define
\begin{align*}\label{e:w-compact}W_{\alpha,a,b}(\xi)=\{g\in G:\Phi_j(g\lambda)&\geq\alpha,\forall(j,\lambda)\in\ccal(\xi)\\&\mbox{and}\h a\leq\Phi_i(g\lambda)\leq b,\forall(i,\lambda)\in\supp(\xi)\}\end{align*}  

We have the following

\begin{prop}\label{w-compact}
Let $\xi$ be an admissible sequence of length $l\geq0.$ Let $\alpha$ and $a<b$ be positive real numbers. Then $W_{\alpha,a,b}(\xi)\Gamma/\Gamma$ is a relatively compact subset of $G/\Gamma.$ 
\end{prop}
\begin{proof}
The same proof as in~\cite[Proposition, 7.14]{GO} goes through.
\end{proof} 

We need a few standard facts about polynomial maps. 
For $N, M, d$ positive integers let us denote by $\pnmd$ the space of polynomials from $K_\nu^N\rightarrow K_\nu^M$ with degree bounded by $d.$ As before we let $\varpi_\nu$ denote the uniformizer of $K_\nu$ and we assume $|\varpi_\nu|_\nu=1/q_\nu.$ 
In non-arthimedean metrics any point of a ball can be considered as the center of the ball thus if $B\subset K_\nu^N$ is a ball and $t\in B$ we have $B=B(t,r)$ for some $r>0$ depending on $B,$ we let $\frac{1}{q_\nu}B(t)$ denote the ball centered at $t$ with radius $r/q_\nu.$ We have

\begin{lem}\label{polynomial1}
Let $B$ be a ball in $K_\nu^N.$ Then for any $F\in\pnmd$ there exists $t_0\in B$ such that $|F(t)|_\nu=\sup_B|F|_\nu$ for all $t\in\frac{1}{q_\nu} B(t_0).$
\end{lem}

\begin{proof}
Let $f$ and $B$ be given and let $t_0\in B$ such that $|F(t_0)|_\nu=\sup_B|F|_\nu.$  Expanding $F$ about $t_0$ we may and will assume $t_0=0.$ Now let $F=(F_1,\ldots,F_M)$ we may and will assume $\sup_B |F|_\nu=|F_1(0)|_\nu.$ We write $f=F_1$ and let 
$$f=a_0+\sum a_{\beta}t_1^{i_1}\cdots t_{N}^{i_N}\h\h\mbox{where}\h\h a_0=f(0)$$ 
The assumption that $|f(0)|_\nu=\sup_B|f|_\nu$ implies that $|a_\beta t_1^{i_1}\cdots t_{N}^{i_N}|_\nu\leq |a_0|_\nu$ for all $t=(t_1,\ldots,t_N)\in B.$ Thus $|a_\beta t_1^{i_1}\cdots t_{N}^{i_N}|_\nu<|a_0|_\nu$ for all $t\in\frac{1}{q_\nu}B.$ Using the ultra metric property, we get $|f(t)|_\nu=|a_0|_\nu$ for all $t\in\frac{1}{q_\nu}B$ as we wanted to show.
\end{proof}

Recall from~\cite[Lemma, 2.4]{KT} that for any $F\in\pnmd,$ any ball $B\subset K_\nu^N$ and any $\vare>0$ we have
\begin{equation}\label{good-func}\theta(\{t\in B:\h |F(t)|_\nu<\vare\sup_B|F|_\nu\})\leq C\vare^{1/Nd}\theta(B)\end{equation}
where $C>0$ is a constant depending only on $N$ and $d.$ As a corollary we get
\begin{lem}\label{polynomial2}
Given $\eta\in(0,1),$ there exists $R>1$ such that if $B_0\subset B$ are two balls with radius $r_0$ and $r$ respectively such that $r_0\geq\eta r,$ then $\sup_B|F|_\nu\leq R\cdot\sup_{B_0}|F|_\nu$ for any $F\in\pnmd.$ 
\end{lem}

Using Lemma~\ref{polynomial1} and Lemma~\ref{polynomial2} the same argument as in~\cite[Proposition 7.12]{GO} gives the following 

\begin{prop}\label{polynomial3}
There exists $R>1$ such that for any $\alpha>0,$ any ball $B$ and any subfamily $\fcal\subset\pnmd$ satisfying
\begin{itemize}
\item[(i)] for any $t_0\in B,$ $\#\{F\in\fcal: |F(t_0)|_\nu<\alpha\}<\infty$\vspace{.5mm}
\item[(ii)] for any $f\in\fcal,$ $\sup_B|F|_\nu\geq\alpha$
\end{itemize}
one of the following holds
\begin{itemize}
\item[(a)] there exists $t_0\in B$ such that $|F(t_0)|_\nu\geq\alpha$ for all $F\in\fcal,$\vspace{.75mm}
\item[(b)] there exists a ball $B_0\subset B$ and $F_0\in\fcal$ such that
$$|F_0(B_0)|_\nu\subset[\alpha/R,R\alpha]\h\h\mbox{and}\h\h \sup_{B_0}|F|_\nu\geq\alpha/R\h\h\mbox{for all}\h\h F\in\fcal$$
\end{itemize}
\end{prop}
The proof of Theorem~\ref{non-div} now goes through the same lines as the proof of Theorem 7.3 in~\cite{GO}, see page 55 in~\cite{GO}.

\section{Proof of Theorem~\ref{horospherical} and equidistribution}\label{sec;horo}
In this section we will prove Theorem~\ref{horospherical}. Here we give the proof assuming Theorem~\ref{thm;exp-mixing}. The proofs indeed go through with non-quantitaive version of decay of matrix coefficients as we remarked above. Before starting the proof wee need the following fact which is essentially a restatement of the fact that $L^2_{00}(X)$ is the orthogonal complement of the space of $G^+$-invariant vectors.  

\begin{lem}\label{G+inv}
Let $\mu$ be a probability measure on $X.$ Suppose that for any $\phi\in L^2_{00}(X)$ we have
$$\int_X\phi(x)d\mu(x)=0$$
Then $\mu$ is $G^+$-invariant
\end{lem}

\begin{proof}
Recall that $L^2_{00}(X)$ is the orthogonal complement of the space of all $G^+$-fixed vectors in $X$ and we have $L^2(X)=L^2_{00}(X)\oplus(L^2_{00}(X))^\perp$ as $G$-representation. Let $g\in G^+$ be arbitrary. For any $\psi\in L^2(X)$ we let $\psi^g(x)=\psi(gx).$ We need to show $\int_X\psi^g(x)d\mu(x)=\int_X\psi(x)d\mu.$ Let $\psi=\psi_0+\psi_1$ where $\psi_0\in L^2_{00}(X)$ and $\psi_1\in(L^2_{00}(X))^\perp.$ Since the decomposition is $G$-equivariant we have $\psi^g=\psi_0^g+\psi_1^g$ is corresponding decomposition of $\psi^g.$ Since $g\in G^+$ we have $\psi_1^g=\psi_1.$ Our hypothesis applied to $\psi_0$ and $\psi_0^g$ implies that
$$\int_X\psi d\mu=\int_X\psi_1 d\mu=\int_X\psi_1^gd\mu=\int_X\psi^gd\mu$$
as we wanted to show. 
\end{proof}

Let us remark that the converse of the above lemma also holds but we do not need the converse here. We now start the

{\bf Proof of Theorem~\ref{horospherical}.} In the proof we will denote  $du=d\theta(u)$ for simplicity.

\textit{Proof of (i).} By our assumption $G/\Gamma$ is compact hence we may take $L=X$ in the statement of Proposition~\ref{m;equidistribution}. Let $\phi\in \cc(X).$ Recall that loc. cit. asserts that 
$$\left|\int_{U_0}\phi(s^nuz)\h du\right|\leq E'e^{-n\kappa'}$$ 
where $E'$ is a constant depending on $\phi$ and $z$ is any point in $X.$ If one takes $z=s^{-n}x$ one then gets
\begin{equation}\label{horo;equi}\left|\int_{U_0}\phi(s^nus^{-n}x)du\right|=\left|\frac{1}{\theta(U_n)}\int_{U_n}\phi(ux)du\right|\leq E'e^{-n\kappa'}\end{equation} 
As we mentioned in section~\ref{sec;notation} the filtration $\{U_m=s^mU_0s^{-m}:\h m\geq0\}$ is an averaging sequence for $U$. Since $\mu$ is $U$-ergodic~\eqref{horo;equi} implies 
\begin{equation}\label{horo;equi2}\int_X\phi d\mu=0\end{equation}

This and Lemma~\ref{G+inv} imply that $\mu$ is $G^+$-invariant. Now since $G^+$ is a normal, unimodular subgroup of $G$ Lemma~\ref{h;normal-unimodular} implies that $\mu$ is the $\overline{G^+\Gamma}$-invariant invariant probability Haar measure on the closed orbit $\overline{G^+\Gamma}/\Gamma$ as we wanted to show.   

\textit{Proof of (ii).} We first show that the assertion holds for $U=U_{\nu},$ a maximal horospherical subgroup of $G_{\nu}.$ 

Let $x\in X$ be $\mu$-generic for the action of $U$ i.e. for any bounded function $\phi\in C^\infty(X)$ and any $\vare>0$ there exists some $n_0$ such that if $n\geq n_0$ then we have 
\begin{equation}\label{birkhoff}\left|\int_{U_0}\phi(s^nus^{-n}x)\h du-\int_X\phi\h d\mu\right|<\vare\end{equation} 

Let $\vare>0$ be given fix some $n_0$ as above and let $L$ be the compact set obtained in Theorem~\ref{non-divergence} for this choice of $\vare$. 
Assume now that $x$ is such that (1) in the Theorem~\ref{non-divergence} holds. Since $s$ normalizes $U$ this implies that (1) in the Theorem~\ref{non-divergence} also holds for $s^{-n}x$ for any $n\in\bbz$. Hence there exists a positive integer $m_0=m_0(n),$ depending on $s^{-n}x,$ such that if $m>m_0,$ then 
\begin{equation}\label{e;nondivergence2}\theta(\{u\in U_m\h:\h us^{-n}x\in L\})\geq (1-\epsilon)\theta(U_m)\end{equation}

For the rest of the argument we may and will assume that $\phi\in\cc(X).$ Let $\mathcal{K}^\ell$ be a deep congruence subgroup such that (i)~\eqref{e;iwasawa-comp} holds for $\kcal^\ell$ (ii) $\phi$ is fixed by $\kcal^\ell$ (iii) for any $x\in L$ the map from $k\mapsto kx$ is injective on $\kcal^{\ell}.$ Fix a large enough $n>n_0$ such that ~(\ref{e;equi'}) holds for $\psi=\frac{1}{\theta(\mathcal{K}^{\ell}\cap U)}\chi_{\mathcal{K}^{\ell}\cap U}$ in the formulation of Proposition~\ref{m;equidistribution} this is to say $E[\mathcal{K}:\mathcal{K}^\ell]s^{-n\kappa'}<\vare.$ Now let $m_0=m_0(n)$ be such that~\eqref{e;nondivergence2} holds for all $m>m_0.$ We have
$$\int_{U_0}\phi(s^ns^mus^{-m}s^{-n}x)du=\int_{U_0}\int_{U_0}\phi(s^ns^mus^{-m}s^{-n}x)du\psi(v)dv=$$
$$\iint_{U_0}\phi(s^ns^m(s^{-m}vs^m)us^{-m}s^{-n}x)\psi(v)dudv=\iint_{U_0}\phi(s^nvs^mus^{-m}s^{-n}x)\psi(v)dudv$$
Note that in the above we used the fact that $s^{-m}vs^m\in U_0$ and that $U_0$ is a group to get the second identity. Let $B\subset U_0$ be such that $u\in B$ implies that $s^mus^{-m}s^{-n}x\in L$ and let  $A=U_0\setminus B.$ By~\eqref{e;nondivergence2} we have $|A|<\vare.$ We have 
\begin{align*}\int_{U_0}\int_{U_0}\phi(s^nvs^mus^{-m}s^{-n}x)\psi(v)dvdu&=\int_{A}\int_{U_0}\phi(s^nvs^mus^{-m}s^{-n}x)\psi(v)dvdu\\ &+\int_{B}\int_{U_0}\phi(s^nvs^mus^{-m}s^{-n}x)\psi(v)dvdu\end{align*} 
The first term is indeed bounded by $\vare\cdot\sup(|\phi|).$ As for the second term recall that since $u\in B$ we may utilize Proposition~\ref{m;equidistribution} and get 
$$\left|\int_{U_0}\phi(s^nvs^mus^{-m}s^{-n}x)\psi(v)dv\right|\leq E[\mathcal{K}:\mathcal{K}^\ell]s^{-n\kappa'}<\vare$$
So there is constant $C'$ depending on $\phi$ such that
$$\left|\int_{U_0}\phi(s^{n+m}us^{-n-m}x)du\right|< C\vare$$
We now use~(\ref{birkhoff}) and get
$$\left|\int_X\phi d\mu\right|< C\vare$$
which thanks to Lemma~\ref{G+inv} says $\mu$ invariant under $G^+$-invariant. Hence again by the Lemma~\ref{h;normal-unimodular} $\mu$ is the $\overline{G^+\Gamma}$-invariant invariant probability Haar measure on the closed orbit $\overline{G^+\Gamma}/\Gamma$ which finishes the proof if $x$ satisfies (1) in Theorem~\ref{non-divergence}.

We will now proceed by induction on the dimension of $\bbg.$ Thanks to the previous paragraph we may and will assume that (2) in Theorem~\ref{non-divergence} holds i.e. if we let $x=g\Gamma,$ then there exists $\mathbb{P}_i$  and $\lambda\in F\Gamma$ such that $g^{-1}Ug\subset \lambda P_i\lambda^{-1}$ where $P_i=\bbp_i(K_S),$ with the notation as in the Theorem~\ref{non-divergence}. Let $\mathbb{P}_i=\mathbb{L}\cdot\mathbb{W}$ be the Levi decomposition of $\bbp_i$ and let $W=\bbw(K_S).$ Let $\mathbb{M}=[\mathbb{L},\mathbb{L}]$ be the derived group of $\mathbb{L}.$ The group $\bbm$ is semisimple and is defined over $K.$ We let  $M=\bbm(K_S).$ Define ${}^{\circ}P_i=MW.$

Note also that $g^{-1}Ug\subset \lambda {}^{\circ}P_i\lambda^{-1}.$ So replacing $U$ by $g^{-1}Ug$ and ${}^\circ P_i$ by $\lambda {}^{\circ}P_i\lambda^{-1}$ we may and will assume that $U\subset {}^{\circ}P_i.$ Since $\lambda\in\bbg(K)$ we have $\Delta=\Gamma\cap {}^{\circ}P_i$ is a lattice in ${}^{\circ}P_i$ and also $W\cap\Gamma$ is a (uniform) lattice in $W$ hence ${}^{\circ}P_i/\Delta$ is a closed subset of $X.$ These reductions also imply that $e\Gamma$ is generic for $\mu$ with respect to the action of $U.$ As a consequence we may and will consider $\mu$ as a measure on ${}^{\circ}P_i/\Delta.$ We write $U=VW$ where $V$ is a maximal unipotent subgroup of $M.$ Using the previous notation we let $W_\nu$, $V_\nu$ and $M_\nu$ be the corresponding $\nu$-components.

Let $\mu=\int_Y\mu_yd\sigma$ be the ergodic decomposition of $\mu$ with respect to $W_\nu.$ Recall that $W_\nu$ is a normal and unimodular subgroup of ${}^{\circ}P_i.$ Hence by Lemma~\ref{h;normal-unimodular} we have that for $\sigma$-a.e. $y\in Y$ the measure $\mu_y$ is the $F=\overline{W_\nu\Delta}$-invariant measure on a closed $F$-orbit where the closure is taken with respect to the Hausdorff topology. Since $\Gamma$ is a congruence lattice lattice and $\bbw$ is a $K$-split unipotent group we deduce that $F=W\Delta.$ Hence for $\sigma$-a.e. $y\in Y$ the measure $\mu_y$ is invariant under $W$. Since $\mu$ is $U_\nu$-invariant ergodic measure we have  $\sigma$ is $V_\nu$-invariant and ergodic measure on $Y.$ Recall that $\mu$ is a measure on ${}^{\circ}P_i/\Delta.$ Now since $\sigma$-a.e. $\mu_y$ is $W$-invariant and $W$ is the unipotent radical of ${}^{\circ}P_i=MW$ we may and will identify $Y$ with $M/\Delta\cap M$ and hence $\sigma$ is a measure on $M/\Delta\cap M.$

Let us recall that $\bbm$ is a semisimple group defined over $K$ and $\Delta\cap M$ is a congruence lattice in $\bbm(\o_S),$ furthermore $\sigma$ is $V_\nu$-invariant ergodic measure on $M/\Delta\cap M.$ Hence we may apply the induction hypothesis and get: There exist some $g\in M$ and a connected $K$-parabolic group $\bbq$ of $\bbm$ such that if we let ${}^\circ Q$ and $Q^+$ be defined as in the introduction then $\sigma$ is the Haar measure on the closed orbit of $g\Sigma/{}^\circ Q\cap\Delta$ where $\Sigma=\overline{Q^+({}^\circ Q\cap\Delta)}$ and the closure is taken with respect to the Hausdorff topology. Hence the measure $\mu$ is the Haar measure on the closed orbit $g\Sigma W/(Q\cap\Delta)(W\cap\Delta).$ Now let $\bbh=\bbq\bbw,$ this is a $K$-parabolic subgroup of $\bbm\bbw.$ We let $H=\bbh(K_S).$ Since $W$ is a subgroup of the split unipotent radical of all minimal $K_S$-pseudo parabolic subgroups of $\bbh$ we get $H^+=Q^+W$ and that $H^+=P^+$ for a $K$-parabolic subgroup $\bbp$ of $\bbg.$ Note also that since $W$ is the unipotent radical and in particular a normal subgroup we have $\Sigma W=\overline{Q^+W({}^\circ Q\cap\Delta)}=\overline{H^+({}^\circ P\cap\Delta)}.$ This finishes the proof in the case $U=U_\nu.$  

We now turn to the general case. Hence $U_\nu\subset U$ and $\mu$ is $U$-invariant ergodic measure on $X.$ Let $\mu=\int_Y\mu_yd\sigma$ be the ergodic decomposition of $\mu$ with respect to $U_{\nu}.$ For $x=g\Gamma\in X$ we let $y(x)$ denote the corresponding point from $(Y,\sigma).$ Since $\mu_{y(x)}$'s are $U_{\nu}$-ergodic the above argument says almost all $\mu_{y(x)}$'s are the $\Sigma(x)$-invariant measure on a closed orbit of $\Sigma(x)\cdot x$ where $\Sigma(x)=g\overline{P^+({}^\circ P\cap\Gamma)}g^{-1},$ and $\bbp$ is a $K$-subgroup. We will say $P$ is associated to $x$. For any $K$-subgroup, $\bbp,$ of $\bbg$ we let $P=\bbp(K_S)$ and let
\begin{equation}\label{e;singular}\scal(P)= \{x\in Y: P\h\mbox{is associated to}\h x\}\end{equation} 
Since there are only countably many such subgroups we have; there exists some $P$ such that $\sigma(\scal(P))>0.$ Note also that $U$ normalizes $U_{\nu},$ so for every $u\in U$ the equality $u\mu_y(x)=\mu_{y(ux)}$ is true for $\mu$-almost all $x\in X.$ Furthermore we have $\Sigma(ux)=u\Sigma(x)u^{-1}.$ Thus $\scal(P)$ is $U$-invariant. This and the fact that $\mu$ is $U$-ergodic imply that $\sigma(\scal(P))=1.$ Observe that $P^+({}^\circ P\cap\Delta)$ is Zariski dense in ${}^\circ P$ hence the map $x\mapsto {\overline{\Sigma(x)}}^z,$ the Zariski closure of $\Sigma(x),$ is a $U$-equivariant Borel map from $\scal(P)$ to $G/N_G({}^\circ P).$ Lemma~\ref{h;bz-measure} now implies that this map is constant almost everywhere. Hence $\mu$ is the $g\overline{P^+({}^\circ P\cap\Gamma)}g^{-1}$-invariant measure on a closed orbit $g\overline{P^+({}^\circ P\cap\Gamma)}\Gamma/\Gamma$ for some $g\in G$. This finishes the proof of Theorem~\ref{horospherical}. 

\textbf{Equidistribution of orbits of horospherical subgroups.} 
As we mentioned in the introduction one application of Theorem~\ref{horospherical} is to prove equidistribution of orbits of $U$ which satisfy the conditions of Theorem~\ref{horospherical}. It is possible to refine the above proof and get equidistribution statement from the proof. In here however we use the linearization technique in order to get equidistribution result. This seems to be shorter and it is a ``standard" method by now. The linearization technique was developed by Dani and Margulis~\cite{DM2} in the Lie group case. Similar linearization statements in the $S$-arithmetic setting in characteristic zero were proved by Tomanov~\cite{T}. In the function field setting such results were proved in~\cite[section 5]{EM}. The linearization technique is an avoidance principle. Roughly speaking it states that unipotent orbits of algebraically generic points do not spend a ``long" time ``too close" to proper algebraic varieties. We have the following

\begin{cor}\label{equi-horo}
Let the notation be as before and assume that either (i) or (ii) of Theorem~\ref{horospherical} is satisfied. Then for any $x=g\Gamma\in X$ and $f\in C_c(X)$ we have
$$\lim_{m\rightarrow\infty}\frac{1}{U_m}\int_{U_m}f(ux)d\theta(u)=\int_Xfd\mu$$
where $\mu$ is the $g\overline{P^+({}^\circ P\cap\Gamma)}g^{-1}$-invariant probability measure on the closed orbit $g\overline{P^+({}^\circ P\cap\Gamma)}g^{-1}\cdot x,$ and $\bbp(K_S)$ is a $K$-parabolic subgroup $\bbp.$   
\end{cor}

\begin{proof}
The proof is standard. Let $x=g\Gamma$ assume $P$ is $K$-parabolic subgroup of $G$ minimal with the property that $g^{-1}Ug\subset {}^\circ P.$ We will show that the theorem holds with this ${}^\circ P.$ Let $\widetilde{X}$ be the one-point compactification of $X$ if $X$ is not compact and be $X$ if $X$ is compact. For any natural number $m$ define the probability measure $\mu_m$ on $X$ by 
$$\int_X f(y)d\mu_m(y)=\frac{1}{U_m}\int_{U_m} f(ux)d\theta(u)$$ 
where $f$ is a bounded continuous function on $X.$ As $\widetilde{X}$ is compact the space of probability measures on $\widetilde{X}$ is weak$^*$ compact. Let $\mu$ be a limit point of $\{\mu_n\}$. By identifying $\mu$ we show that there is only one limit points which in return gives convergence. Note that since $Y={}^\circ P/({}^\circ P\cap\Gamma)$ is closed we may and will assume that $\mu_m$'s and also $\mu$ are supported on $Y\cap\{\infty\}.$
It follows from nondivergence of unipotent trajectories, see Theorem~\ref{kleinbock-margulis}, that $\mu$ is concentrated on $Y$. Note that $\mu$ is $U$-invariant. We let $\mu=\int_Y\mu_yd\sigma(y)$ be a decomposition of $\mu$ into $U$-ergodic components. For any proper $K$-parabolic subgroup $F$ of $P$ let 
$$\scal(F)=\{p\in{}^\circ P:\h p^{-1}Up\subset F\}$$
By Theorem~\ref{horospherical} each ergodic component is supported on some $\scal(F).$ Minimality of $P$ together with theorem 5.5 in~\cite{EM} implies $\mu(\scal(F))=0$ for all proper subgroups $F.$ Since there are only countably many such $F$ we get that $\sigma$-almost all $\mu_{y(z)}$'s are the $P^+$-invariant measure on the closed orbit $g_z\overline{P^+({}^\circ P\cap\Gamma)}g_z^{-1}\cdot z.$ Now since $g^{-1}Ug\subset P^+$ we see that the support of $\mu$ is in $g\overline{P^+({}^\circ P\cap\Gamma)}g^{-1}\cdot x$ which then implies the theorem.
\end{proof}


\end{document}